\documentclass[a4paper,12pt]{article}
\usepackage{caption}
\usepackage{subcaption}

\usepackage{color}
\usepackage{graphicx}
\graphicspath{{images/}}
\usepackage{latexsym}
\usepackage{amsmath}
\usepackage{amssymb}
\usepackage{enumerate}
\usepackage{pict2e}
\usepackage{theorem}
\newtheorem{theorem}{Theorem}[section]

\newtheorem{corollary}[theorem]{Corollary}

\theorembodyfont{\rmfamily}
\newtheorem{proof}{\textmd{\textit{Proof.}}}

\makeatletter

\@addtoreset{equation}{section}
\makeatother

\date{}
\begin{document}
\begin{center}
{\large{\bf SURFACES OF REVOLUTION ADMITTING STRONGLY CONVEX SLOPE METRICS}}
\end{center}
\medskip







\begin{center}
{\bf Pattrawut Chansangiam, Pipatpong Chansri
and Sorin V. Sabau}
\end{center}
\let\thefootnote\relax\footnote{Bulletin of the Transilvania University of Brasov Vol 13(62), No. 1, 2020}\\
\let\thefootnote\relax\footnote{Series III: Mathematics, Informatics, Physics, pp. 77-88} 
\pagestyle{myheadings}
\begin{abstract}
  This paper discusses the geometry of a surface endowed with a slope metric. We obtain necessary and sufficient conditions for any surface of revolution to admit a strongly convex slope metric. Such conditions involve certain inequalities for the derivative of the associated function on the Cartesian coordinate and the polar coordinate. In particular, we apply this result to certain well-know surface of revolution.


\end{abstract}
\section{Introduction and Preliminaries}
\hskip0.6cm
Optimal transport and optimal control problems are important topics in pure and applied mathematics. For instance, it is well-known that the minimizing time travel in the Euclidean space is the straightline. Moreover, when traveling from a point $A$ to a point $B$ in the Euclidean space, the distance and hence the time is the same as when traveling from $B$ to $A$.

However, in real life problems, applications to engineering, industry, etc., there are external forces acting on the traveling object in Euclidean space, like magnetic fields, gravitational fields, winds, etc. 
For the sake of simplicity, let us assume we need to travel by a ship in an open sea from a point $A$ to a point $B$, say from the pier to an island. In the absence of any wind or marine currents the shortest traveling time is the straighline, however if a mild wind comes up, then we arrive to the following control problem.

{\it 
Consider a ship sailing in the open sea in calm waters. Suppose a mild wind comes up. How must the ship be steered in order to reach a given destination in the shortest time?
}

This problem was considered for the first time by E. Zermelo in 1931 (\cite{Z}) when he assumed that the open sea was the Euclidean plane $\mathbb{R}^2$ with the Euclidean metric.

The first thing to remark is that, when travelling with constant speed under the action of a mild wind, the minimial time needed to travel  from $A$ to $B$ is different from the time needed to travel from $B$ to $A$. Indeed, the time minimal paths and the travel times are different when sailing against the wind and when sailing in the same direction with the wind.  This is called nowadays the {\it Zermelo's navigation problem} (see \cite{BR} for details and generalizations). 

The simple example above suggests that there are other types of distances and metrics, called Randers metrics, that give different minimial time trajectories from the canonical Euclidean metric. Randers metrics are widely used in optimal control, Physics, Biology and many other fields of pure and applied mathematics (\cite{AIM}, \cite{SSS}). Randers metrics belong to a more general family of metrics, called Finsler metrics (see \cite{AIM}, \cite{BCS} for basics on Finsler metrics).

This paper has no intention to be an introduction to Finsler metrics. We will recall only that {\it Finsler geometry is just the Riemannian geometry without the quadratic restriction} (\cite{C}). Indeed, what we call today a Finsler norm was actually introduced by B. Riemann in his famous Habilitation disertation from 1854, namely a metric function
$$
ds=F(x^1,x^2,\dots,x^n;dx^1,dx^2,\dots,dx^n)
$$
that depends on the position $(x^1,x^2,\dots,x^n)$ of the point 
and direction. Such a metric is determined by a function $F$ defined on the tangent bundle $TM$ of an $n$-dimensional smooth manifold $M$ having the properties:
\begin{enumerate}
\item $F(x,y)$ is positive on $TM\setminus\{0\}$;
\item $F(x,y)$ it is 1-homogeneous in $y$, that is $F(x,\lambda y)=\lambda F(x,y)$, for all constants $\lambda>0$;
\item the Hessian matrix
  \begin{equation}\label{Hess F2}
  g_{ij}=\dfrac{1}{2}\dfrac{\partial F^2(x,y)}{\partial y^i\partial y^j}
\end{equation}
 is positive definite on $TM\setminus\{0\}$.
  \end{enumerate}

  The most important special case is the case when $F^2=g_{ij}(x)y^iy^j$, in other words $F$ is a quadratic form in the variable $y$. This is what we call today a {\it Riemannian metric}. Therefore, rather than regarding the Finsler geometry as just a generalization of Riemannian geometry, it is more appropiate to say that Finsler geometry is just the Riemannian geometry without the quadratic restriction above.

  The Randers metrics naturally appearing from the Zermelo's navigation problem are deformations of Riemannian metrics in the sense $F=\alpha+\beta$, where $\alpha=\sqrt{a_{ij}(x)y^iy^j}$, and $\beta=b^i(x)y^i$, where $a_{ij}$ is a Riemannian metric.

    Another type of Finsler metric is the so-called {\it slope metric} introduced by M. Matsumoto in  1989 (see \cite{M1}) based on a letter of P. Finsler, the foundator of Finsler geometry. The control problem is the following

  {\it  Suppose a person walking on a horizontal plane with velocity c, while the gravitational
force is acting perpendicularly on this plane. The person is almost ignorant of the action of
this force. Imagine the person walks now with same velocity on the inclined plane of angle
$\varepsilon$ to the horizontal sea level. Under the influence of gravitational forces, what is the trajectory the person should walk in the center to reach a given destination in the shortest time?
}

The metric giving the shortest travelling paths was called therefore the {\it slope metric} (\cite{M1}, \cite{M2}). Obviously the problem described above it is important for applications to the real world in industry, constructions, or when transporting somethig on an a slope.  It was shown by M. Matsumoto (see \cite{M1}, \cite{M2}) that the corresponding Finsler metric is also a deformation of a Riemannian metric $\alpha$ by a linar form $\beta$ by the formula
$$
F=\frac{\alpha^2}{\alpha-\beta}.
$$

Another important reason to consider the slope metrics is the fight against wild fires (\cite{Mv}). 
Almost everyday we hear news about wildfires in different regions of the Earth. In order to deal with it, the firefighters need to act promptly based on a proper understanding, predicting and modelling the evolution of wildfires. The slope metric is one of he most appropiate mathematical models for predicting the evolution of wildfires, hence the mathetical study of such metrics might become in the future vital for predicting, controling and  successfully fighting the wildfires (see \cite{Mv} and references within for details).

Let us remind that the study of shortest paths in a Riemannian or Finsler space is related to the calculus of variations. A variational problem determined by a Finsler metric $F$ is non-degenarate when the Hessian  \eqref{Hess F2} is regular matrix. Moreover, the extremal paths are minimizing when the metric is positive definite, and this happens in the case when the unit circle in the tangent space $T_xM$ is strongly convex (see \cite{BCS} for a very detailed exposition of the calculus of variations for Finsler spaces).

Based on the discussion above, we can formulate the problem we are going to study in the present paper.

{\it 
The wildfires can happen across geographically very complicated terrain including forests, glasslands, rice fields, etc. On what kind of surfaces, the wildfire evolution can be modeled as the variational problem of the slope metric $F=\alpha^2/(\alpha-\beta)$? In other words, on which surfaces there are naturally induced slope metrics whose Hessian matrices \eqref{Hess F2} are positive defined?  
}

In the present paper we will give some general conditons for a slope metric to have positive definite Hessian and we will give some examples of inclined surfaces where the wildfires behaviour can be predicted and eventually controled.

\begin{figure}[h]
    \centering
   
    \begin{subfigure}[b]{0.4\textwidth}
    \centering
        \includegraphics[width=\textwidth]{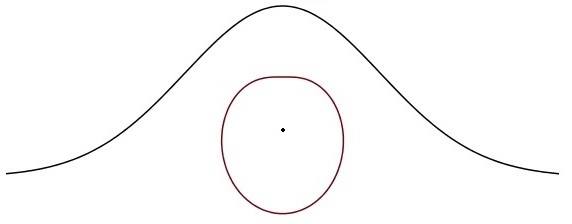}
        \caption{The unit circle of the slope metric in the case of walking in the slope of a mountain.}
    \end{subfigure}\quad
          \begin{subfigure}[b]{0.4\textwidth}
          \centering
        \includegraphics[width=\textwidth]{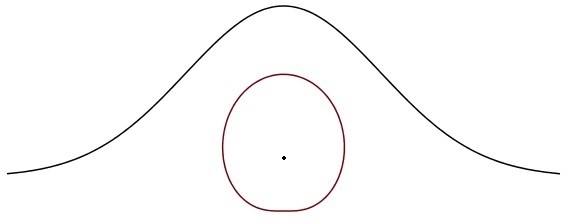}
        \caption{the slope of a wildfire on slope of a mountain.\\}
        \end{subfigure}
    \caption{}
\end{figure}





Let us formulate the above problem in the Riemannian geometry formalism. Consider a surface $S$ embedded in the Euclidean space $\mathbb{R}^3$ parametrized by
		\begin{equation}\label{surface parametrized}
		S \to \mathbb{R}^3,(x,y)\mapsto (x,y,z=f(x,y)),
		\end{equation}
for a smooth function $f$, that is $S$ is the graph of the function $z=f(x,y)$. It is well known that $S$ is a $2$-dimentional differential manifold. It follows that the induced Riemannian metric from $\mathbb{R}^3$ on the surface $S$ is represented by the matrix
		\begin{equation*} a_{ij}=
		\left( \begin{array}{cc}
			1+f_x^2&f_xf_y\\
			f_xf_y&1+f_y^2
		\end{array} \right),
		\end{equation*}
where $f_x$ and $f_y$ are the partials derivative of  $f$ with respect to $x$ and $y$, respectively.
We consider the tangent plane $\pi_p$ spanned by two vectors
		\begin{equation*}
		\partial_x:=(1,0,f_x),\quad \partial_y:=(0,1,f_y).
		\end{equation*}
We can construct an $a$-orthogonal basis ${e_1,e_2}$ in $\pi_p$ by choosing $e_1$ to point on the steepest downhill direction of $\pi_p$.
In the plane $\pi_p$ with origin $p=(0,0)$ and $(X,Y)$ in the basis $\{e_1,e_2\}$, the slope principle by Matsumoto \cite{M1} tells us that the time minimizing trajectory of a hiker on the plane $\pi_p$ is given by walking in the direction given by the lima\c{c}on 
		\begin{equation}
		 r=v+a\cdot \cos\theta,
		\end{equation}		
where $(r,\theta)$ are the polar coordinate of the $XY$ plane, $v$ is the velocity of the hiker on the flat plane $xy$, $a=\frac{g}{2}\cdot \sin\varepsilon$ where $g$ is gravitational constant for fixed $p\in S$ and $\varepsilon$ is the incline angle to the sea level. Moreover any vector of $\pi_p$ can be written as $\dot{x}\partial_x+\dot{y}\partial_y$. By using the relation between the coordinates $(X,Y)$ in the basis $\{e_1,e_2\}$ and the canonical component $(\dot{x},\dot{y})$ we obtain the implicit lima\c{c}on equation $h(\dot{x},\dot{y})=0$ where
		\begin{equation*}	
		h(\dot{x},\dot{y})=\dot{x}^2+\dot{y}^2+(f_x\dot{x}+f_y\dot{y})^2-v\cdot\sqrt{\dot{x}^2+\dot{y}^2+(f_x\dot{x}+f_y\dot{y})^2}-\frac{g}{2}\cdot(f_x\dot{x}+f_y\dot{y}).
		\end{equation*}
From this, by use Okubo's method (see \cite{AIM} for detail), we can describe the surface $S$ via the {\it fundamental function} $F:TS\rightarrow [0,\infty)$ such that
$$h(\frac{\dot{x}}{F},\frac{\dot{y}}{F})=0.$$
By solving this equation we received
\begin{equation*}	
		\begin{split}
		F=\frac{\alpha^2}{v\alpha-\frac{g}{2}\beta},
		\end{split}		
		\end{equation*}
where
\begin{equation}\begin{cases}
	\alpha&=\sqrt{a_{11}\dot{x}^2+2a_{12}\dot{x}\dot{y}+a_{22}\dot{y}^2}\ \text{is the induced Riemannian metric on}\ S,\ \text{and}\\
	\beta&= f_x\dot{x}+f_y\dot{y}.
	\end{cases}
	\end{equation}
By normalization, we get the usual form of the {\it slope metric}
		\begin{equation}\label{e1}	
		\begin{split}
		F=\frac{\alpha^2}{\alpha-\beta}.
		\end{split}
		\end{equation}
It is well known that the function $F$ induces a Finsler norm on the surface $S$ if and only if the $h$ is a convex (see \cite{BCS} for detail). In this case $(S,F)$ becomes a Finsler manifold. The next theorem provides a necessity and sufficiency condition for $h$ to be strongly convex.

\begin{theorem}\label{thm1}\cite{SS}
A $2$-dimensional differential manifold $M$ endowed with the fundamental function \eqref{e1} is a Finsler manifold if and only if $f_x^2+f_y^2<\frac{1}{3}$.
\end{theorem}

In the present work, we will focus on a well known class of $2$-dimensional differential manifolds generated by rotating a curve (called the profile curve) around an axis of rotation, called a surface of revolution. Indeed, for each smooth function $\phi:[0,\infty)\rightarrow\mathbb{R}$ that is extensible to a smooth even function around $0$, the surface of revolution in the $3$-dimensional Euclidean space is defined by the equation $z=\phi(\sqrt{x^2+y^2})$.


There are several ways of introducing local coordinates on such a surface. The parametrization \eqref{surface parametrized} of an arbitrary surface can be adapted to the case of surfaces of revolution by putting
\begin{equation}\label{surface of revolution parametrized}
S\to\mathbb{R}^3,\quad (x,y)\mapsto (x,y,z=\phi(\sqrt{x^2+y^2}))
\end{equation}
The induced Riemannian metric is
		\begin{equation*}\begin{split}a_{ij}(x,y)&:=\left( \begin{array}{cc}
				1+\phi_x^2&\phi_x\phi_y\\
				\phi_x\phi_y&1+\phi_y^2
			\end{array} \right)\\
			&=\left( \begin{array}{cc}
				1+\frac{\left(\phi'(\sqrt{x^2+y^2})\right)^2x^2}{x^2+y^2}&\frac{\left(\phi'(\sqrt{x^2+y^2})\right)^2xy}{x^2+y^2}\\
				\frac{\left(\phi'(\sqrt{x^2+y^2})\right)^2xy}{x^2+y^2}&1+\frac{\left(\phi'(\sqrt{x^2+y^2})\right)^2y^2}{x^2+y^2}
			\end{array} \right),	
			\end{split}
		\end{equation*}
provide $x^2+y^2\neq 0$, where 
$$\phi_{x}=\frac{\phi'(\sqrt{x^2+y^2})x}{\sqrt{x^2+y^2}}\quad \text{and}\quad \phi_{y}=\frac{\phi'(\sqrt{x^2+y^2})y}{\sqrt{x^2+y^2}},$$ 
and $\beta=\phi_x\dot{x}+\phi_y\dot{y}$, then we can define the slope metric $F$ for surface of revolution in term of function $\phi$. 

Our main work in this paper is to investigate necessary and sufficient conditions for any surface of revolution to admit a strongly convex slope metric. Such conditions rely on certain inequalities for the derivative of the profile curve on the Cartesian coordinate and the polar coordinate. In particular, we apply this result to classical well-known surfaces of revolution.


\section{Convexity conditions for the slope metric on surface of revolution}
\hskip0.6cm 
\begin{theorem}\label{main theo}
Let $S\to\mathbb{R}^3$ be a surface of revolution. Then the following statements are equivalent:
\begin{itemize}
\item[(i)] $S$ admits a strongly convex slope metric;
\item[(ii)] $[\phi'(s)]^2<\frac{1}{3}$ where $S\to\mathbb{R}^3:(x,y)\mapsto (x,y,z=\phi(s))$ and $s=\sqrt{x^2+y^2}$;
\item[(iii)] $[m'(u)]^2>3$ where $S\to\mathbb{R}^3:(u,v)\mapsto (m(u)\cos v,m(u)\sin v,u)$.
\end{itemize}
\end{theorem}
\begin{proof} 
First, we shall prove $(i)\Leftrightarrow (ii)$. 
Note that for the surface of revolution $S:z=\phi(\sqrt{x^2+y^2})$, we have
\begin{equation*}
\begin{split}
(\phi_x)^2+(\phi_y)^2&=\left(\phi'(\sqrt{x^2+y^2})\frac{x}{\sqrt{x^2+y^2}}\right)^2+\left(\phi'(\sqrt{x^2+y^2})\frac{y}{\sqrt{x^2+y^2}}\right)^2\\
&=\left(\phi'(\sqrt{x^2+y^2})\right)^2.
\end{split}
\end{equation*}
Recall from Theorem \ref{thm1} that an equivalent condition for $S$ to admit a strongly convex slope metric is the inequality
\begin{equation*}
(\phi_x)^2+(\phi_y)^2<\frac{1}{3}.
\end{equation*}
Thus the assertions $(i)$ and $(ii)$ are equivalent.

Next, we shall show the equivalence between $(ii)$ and $(iii)$. Indeed, we can parametrize the surface of revolution by using trigonometric functions, that is
\begin{equation}\label{surf_uv}
S\to\mathbb{R}^3,\quad (x=m(u)\cdot\cos v,y=m(u)\cdot\sin v,z=u).
\end{equation}
This surface of revolution is obtained by rotating the curve $x=m(z)$ along the $z$ axis. Applying the parametrization \eqref{surface of revolution parametrized} to \eqref{surf_uv} yields
\begin{equation*}
x^2+y^2=m^2(u)\cos^2v+m^2(u)\sin^2v=m^2(u),
\end{equation*}
and hence
\begin{equation*}
z=\phi(\sqrt{x^2+y^2})=\phi(| m(u)|).
\end{equation*}
For the sake of simplicity we consider here only the case $+m(u)$. Since $z=u$ it follows that $\phi(m(u))=u$. 
Substituting $s=\sqrt{x^2+y^2}$ yields $s=m(u)$ and thus
$$m(\phi(s))=m(\phi(m(u)))=m(u)=s.$$
Hence, $\phi$ and $m$ are inverse function of each other. From  the inverse function theorem, we have
\begin{equation*}
m'(u)=\frac{1}{\phi'(s)}\Bigg\vert_{s=m(u)}
\end{equation*}
or, equivalently

\begin{equation*}
[m^{-1}]'(s)=\frac{1}{m'(m^{-1}(s))}\Leftrightarrow \phi'(s)=\frac{1}{m'(u)}\Bigg\vert_{u=\phi(s)}
\end{equation*}
and from here it is clear that
\begin{equation*}
[m'(u)]^2=\left[\frac{1}{\phi'(s)}\right]^2.
\end{equation*}
The condition $[\phi'(s)]^2<\frac{1}{3}$ is now equivalent to $[m'(u)]^2>3$.
\end{proof}

\section{Classical surfaces of revolution admitting strongly convex slope metrics}

In this section, we investigate several classical surfaces of revolution to admit  convex slope metrics. Here are our main worked out examples.

\subsection{The Elliptic paraboloid}
Consider the elliptic paraboloid $S\to\mathbb{R}^3$, $(x,y)\mapsto (x,y,z=f(x,y)=100-x^2-y^2)$ this surface was firstly studied by Bao, Robles (see \cite{BR}). In this parametrization, the convexity condition is $f_x^2+f_y^2=4(x^2+y^2)<1/3$, i.e. $x^2+y^2< 1/12$, hence the strongly convexity condition is satisfied only on a circular vicinity of height $1/12$ units of the hilltop.

\begin{figure}[h]
\begin{center}

\setlength{\unitlength}{0.75cm}
\begin{picture}(5,6)

\put(0,1.5){\vector(1,0){5}}
\put(2.5,0){\vector(0,1){6}}
\put(3,2){\vector(-1,-1){2}}

\put(3.7,4){\line(-1,1){1}}
\put(3.8,3.5){\line(-1,1){1.5}}
\put(3.5,3.4){\line(-1,1){1.4}}
\put(3.2,3.3){\line(-1,1){1.35}}
\put(2.9,3.2){\line(-1,1){1.3}}
\put(2.6,3.1){\line(-1,1){1.2}}
\put(2.15,3.15){\line(-1,1){0.8}}

\put(1.7,3.3){\line(-1,1){0.45}}

\qbezier(2.5,5)(4,5)(4.2,1)
\qbezier(2.5,5)(1,5)(0.8,1)

\qbezier(2.5,3.1)(3.4,3.1)(3.85,3.5)
\qbezier(2.5,3.1)(1.6,3.1)(1.15,3.5)

\qbezier[25](2.5,3.9)(3.4,3.9)(3.85,3.5)
\qbezier[25](2.5,3.9)(1.6,3.9)(1.15,3.5)

\put(2,2.7){$100-\frac{1}{12}$}
\put(2.5,5.1){$100$}
\put(0,1){$10$}
\put(4,1){$-10$}
\put(2.6,1){$0$}
\put(5,1.3){$x$}
\put(1.3,0){$y$}
\put(2.4,6.2){$z$}
\qbezier(3,0.5)(2.5,0.2)(2,0.5)
\put(3.1,0.59){\vector(1,1){0}}
\end{picture}
\caption{}
\end{center}
\end{figure}
Putting this surface into the parametrization \eqref{surface of revolution parametrized}, we have $z=\phi(s)$ where $\phi(s)=100-s^2$. The strongly convexity condition $[\phi'(s)]^2<1/3$ is equivalent to 
$-1/{\sqrt{12}}<s< 1/{\sqrt{12}}$.

\begin{figure}[h]
\begin{center}

\setlength{\unitlength}{0.75cm}
\begin{picture}(5,7)

\put(0,1.5){\vector(1,0){5}}
\put(2.5,0){\vector(0,1){6}}

\put(3.7,3.9){\line(-1,1){1.1}}
\put(3.8,3.5){\line(-1,1){1.5}}
\put(3.5,3.5){\line(-1,1){1.4}}
\put(3.2,3.5){\line(-1,1){1.3}}
\put(2.9,3.5){\line(-1,1){1.15}}
\put(2.6,3.5){\line(-1,1){1}}
\put(2.15,3.5){\line(-1,1){0.7}}

\put(1.7,3.5){\line(-1,1){0.43}}

\qbezier(2.5,5)(4,5)(4.2,1)
\qbezier(2.5,5)(1,5)(0.8,1)

\qbezier(3,0.5)(2.5,0.2)(2,0.5)
\put(3.1,0.59){\vector(1,1){0}}

\qbezier(1.15,3.5)(2.5,3.5)(3.85,3.5)

\put(4.2,2.7){$\phi(s)=100-s^2$}

\put(5,1.3){$S$}

\put(2.4,6.2){$z$}

\end{picture}
\caption{}
\end{center}
\end{figure}

Finally, in the trigonometric parametrization we have
\begin{equation*}
(u,v)\mapsto (x=m(u)\cos v,y=m(u)\sin v,z=u)\; {for}\;  u\in (-\infty,100],\; v\in[0,2\pi).
\end{equation*}
This surface of revolution is obtained by rotating the curve $x=m(z)$ around the $z$ axis. The inverse of 
$\phi$ is given by the function $s=\pm\sqrt{100-u}$ for $u\in(-\infty,100]$. For simplicity, we consider only $s=\sqrt{100-u}$.\begin{figure}[h]
\begin{center}
\setlength{\unitlength}{0.75cm}
\begin{picture}(5,6)
\put(-2,3){\vector(1,0){10}}
\put(0,0){\vector(0,1){6}}
\put(3,1.6){\line(0,1){2.8}}
\put(3.5,1.7){\line(-1,1){0.5}}
\put(3.8,1.9){\line(-1,1){0.8}}
\put(4.2,2.1){\line(-1,1){1.2}}
\put(4.6,2.3){\line(-1,1){1.5}}
\put(4.8,2.65){\line(-1,1){1.6}}
\put(4.8,3.05){\line(-1,1){1.2}}
\put(-1.35,4){\vector(1,1){0}}
\qbezier(-1.4,4)(-2,3)(-1.4,2)
\qbezier(4.8,3)(4.8,4.8)(-1.5,4.8)
\qbezier(4.8,3)(4.8,1.2)(-1.5,1.2)
\put(5,3.2){$100$}
\put(3.2,1){$m(u)=-\sqrt{100-u}$}
\put(0.1,6){$s$}
\put(2,3.2){$100-\frac{1}{12}$}
\put(8,3.2){$u$}
\end{picture}
\caption{}
\end{center}
\end{figure}
We have $[m'(u)]^2= 1/(4(100-u))>3$, or $100-(1/12)<u\leq 100$, i.e. the convexity domain is the same as before.

\begin{corollary}
For the elliptic paraboloid $z=\phi(s)=100-s^2$, its slope metric \eqref{e1}  is strongly convex on the open domain $\{(x,y)\in S \mid \sqrt{x^2+y^2}< 1/\sqrt{12}\}$.
\end{corollary}

\subsection{The Cone}
Consider a cone, that is, a surface of revolution defined by
\begin{align}
S\to \mathbb{R}^3,\quad (x,y)\to \left(x,y,z =a\sqrt{x^2+y^2}\right),\quad \text{for}\ a>0 \label{cone1}. 
\end{align}
We can write $z = \phi(s):=as$.
Since $\phi'(s)=a$, by Theorem \ref{main theo} its slope metric is strongly convex on the whole surface if and only if $a\in (0,1/\sqrt{3})$. We obtain 
\begin{corollary}
The slope metric \eqref{e1} defined on the cone $S$ given by \eqref{cone1} is strongly convex if and only if 
$a\in (0,1/\sqrt{3})$.
\end{corollary}

\subsection{The Ellipsoid}
Consider a Ellipsoid, that is, a surface of revolution defined by
\begin{align}
(x,y)\to \left(x,y,z =\frac{c}{a}\sqrt{a^2-x^2-y^2}\right),\quad \text{for}\ a>0. \label{Ellipsoid}  
\end{align}
Write $z = \phi(s):=c\sqrt{a^2-s^2} / a$.
Then $\phi'(s)= - cs /(a\sqrt{a^2-s^2})$, and thus the strongly convexity condition is satisfied only for $s\in(-\frac{a^2}{\sqrt{a^2+3c^2}},\frac{a^2}{\sqrt{a^2+3c^2}})$.
\begin{corollary}
The slope metric \eqref{e1} defined on the ellipsoid $S$ given by \eqref{Ellipsoid} is strongly convex on the open domain $\{(x,y)\in S \mid \sqrt{x^2+y^2}<\frac{a^2}{\sqrt{a^2+3c^2}} \}$.
\end{corollary}


\subsection{The Two-sheeted hyperboloid}
Consider the surface of revolution
\begin{equation}\label{Two-sheeted hyperboloid}
(x,y)\to \left(x,y,z =a\sqrt{x^2+y^2+b^2}\right),\quad \text{for}\ a>0.
\end{equation}
We can write $z=\phi(s) :=a\sqrt{s^2+b^2}$. 
Then $\phi'(s)= as /(\sqrt{s^2+b^2})$, and thus
 the strongly convexity condition of the slope metric for this surface is $a\in (0,1/\sqrt{3})$.
\begin{corollary}
The slope metric \eqref{e1} defined on the hyperboloid $S$ given by \eqref{Two-sheeted hyperboloid} is strongly convex if and only if  
$a\in (0,1/\sqrt{3})$.
\end{corollary}

\subsection{The One-sheeted hyperboloid}
Consider the surface of revolution
\begin{equation}\label{One-sheeted hyperboloid}
(x,y)\to \left(x,y,z =a\sqrt{x^2+y^2-b^2}\right),\quad \text{for}\ a>0.
\end{equation}
Putting $z=\phi(s):=a\sqrt{s^2-b^2}$, 
we get $\phi'(s)= as/(\sqrt{s^2-b^2})$. Thus, the strongly convexity condition is satisfied only for $s\in(-\infty,-\frac{|b|}{\sqrt{1-3a^2}})\cup(\frac{|b|}{\sqrt{1-3a^2}},\infty)$ where $1-3a^2>0$.
 
\begin{corollary}
The slope metric \eqref{e1} defined on the hyperboloid $S$ given by \eqref{One-sheeted hyperboloid} is strongly convex on the open domain $\{(x,y)\in S \mid \sqrt{x^2+y^2}> |b| / \sqrt{1-3a^2} \}$, where $0<a<1/\sqrt{3}$.
\end{corollary}

\subsection{The surface of revolution generated by the Gauss function}

Recall that the Gauss function $h(x)=e^{-x^2}$ is frequently used in statistics and it gives the shape of the normal distribution. Let us consider the surface of revolution

\begin{equation}\label{Gauss function}
(x,y)\to \left(x,y,z=f(x,y)=\frac{1}{2\sqrt{6}}e^{-x^2-y^2}\right).
\end{equation}
We can write $z=\phi(s):= e^{-s^2} / (2 \sqrt{6})$. 
Since $se^{-s^2}<1$ for any $s>0$, the slope metric is globally strong convex on this surface.

\begin{figure}[h]
\begin{center}

\setlength{\unitlength}{0.75cm}
\begin{picture}(5,7)

\put(0,1){\vector(1,0){5}}
\put(2.5,0){\vector(0,1){6}}

\qbezier(4.4,2.5)(4.8,1.2)(6,1.2)
\qbezier(2.5,5)(4,5)(4.4,2.5)
\qbezier(2.5,5)(1,5)(0.6,2.5)
\qbezier(0.6,2.5)(0.2,1.2)(-1,1.2)

\put(3,5.1){$\frac{1}{2\sqrt{6}}\simeq 0.2$}
\put(2.6,0.5){$0$}

\qbezier(3,0.5)(2.5,0.2)(2,0.5)
\put(3.1,0.59){\vector(1,1){0}}
\end{picture}
\caption{}
\end{center}
\end{figure}

Observe that $\phi(s)$ is bijective for $s\in(0,\infty)$, and by putting $\phi(s)=u$ and we get

\begin{equation*}
m(u)=\pm \frac{\sqrt{-2\ln(24u^2)}}{2}\quad \text{for } u\in \left(0,\frac{1}{2\sqrt{6}}\right),
\end{equation*}
\begin{figure}[h]
\begin{center}

\setlength{\unitlength}{1cm}
\begin{picture}(5,7)

\put(-2,3){\vector(1,0){10}}
\put(0,0){\vector(0,1){6}}

\put(-1.35,4){\vector(1,1){0}}
\qbezier(-1.4,4)(-2,3)(-1.4,2)

\qbezier(1.6,4.5)(0.5,4.8)(0.5,5.8)
\qbezier(5.5,3)(5.5,4.5)(1.6,4.5)
\qbezier(5.5,3)(5.5,1.5)(1.6,1.5)
\qbezier(1.6,1.5)(0.5,1.3)(0.5,0.2)

\put(3.2,1){$m(u)=-\frac{\sqrt{-2\ln(24u^2)}}{2}$}
\put(3.2,5){$m(u)=\frac{\sqrt{-2\ln(24u^2)}}{2}$}
\put(0.1,6){$s$}
\put(8,3.2){$u$}

\end{picture}
\caption{}
\end{center}
\end{figure}
We denote $\mu(u)=[m'(u)]^2$  for $(0,1/(2\sqrt{6}))$ and compute the minimum of $\mu(u)$. The only critical point of $\mu$ in   $(0,1/(2\sqrt{6}))$ is $e^{-\frac{1}{2}}/(2 \sqrt{6})$. Since
\begin{equation*}
\mu\left(\frac{1}{2\sqrt{6}}e^{-\frac{1}{2}}\right)\approx 32.6>3,
\end{equation*}
the convexity condition is globally verified.

\begin{corollary}
For the surface of revolution 
$z=\phi(s) = e^{-s^2} / (2 \sqrt{6})$, 
its slope metric \eqref{e1} is globally strongly convex on a surface. 
\end{corollary}

\vskip.5cm \noindent{\bf Acknowledgement(s) :} We thank Professor Hideo Shimada from Tokai university, who provided insight and expertise that greatly assisted the research. The second author would like to thank King Mongkut's Institute of Technology Ladkrabang Research Fund for financial supports, grant no. KREF046201.

\end{document}